\documentclass{article}
\usepackage{amsmath}
\usepackage{amssymb}

\newtheorem{theorem}{Theorem}

\newtheorem{lemma}[theorem]{Lemma}

\newtheorem{remark}[theorem]{Remark}

\newenvironment{proof}[1][Proof]{\noindent\textbf{#1.} }{\ \rule{0.5em}{0.5em}}
\input{tcilatex}
\begin{document}

\title{The remainder in the Renewal Theorem}
\author{Ron Doney}
\date{}

\begin{abstract}
If the step distribution in a renewal process has finite mean and regularly
varying tail with index $-\alpha ,$ $1<\alpha <2,$ the first two terms in
the asymptotic expansion of the renewal function have been known for many
years. Here we show that, without making any additional assumptions, it is
possible to give, in all cases except for $\alpha =3/2$ $,$ the exact
asymptotic behaviour of the next term. In the case $\alpha =3/2$ the result
is exact to within a slowly varying correction. Similar results are shown to
hold in the random walk case.
\end{abstract}

\maketitle

\section{Introduction and Results}

We consider a renewal process $(S_{n},n\geq 0)$, i.e. a random walk with
non-negative, i.i.d increments $X_{1,}X_{2},\cdots $ with a distribution $F$
whose tail $\overline{F}$ $\in RV(-\alpha )$ (i.e. is regularly varying at
infinity with index $-\alpha $) where $\alpha \in (1,2]$ and we assume $%
\int_{0}^{\infty }$ $y^{2}dF(y)=\infty $ if $\alpha =2.$ We write $EX_{1}=m$
and define a distribution $\Phi $ via its density function%
\begin{equation}
\phi (y)=\frac{P(X_{1}>y)}{m}:=m^{-1}\overline{F}(y),\text{ }y\geq 0,\text{%
and write }\overline{\Phi }(x)=\int_{x}^{\infty }\phi (y)dy.  \label{1}
\end{equation}%
The object of our study is the renewal function $U(x):=U([0,x]),$ where the
renewal measure is defined by 
\begin{equation}
U(dx):=\sum_{0}^{\infty }P(S_{n}\in dx),  \label{4}
\end{equation}%
with $S_{0}\equiv 0.$ Since $\Phi $ is the limiting and stationary
distribution in the process of overshoots in $S,$ its importance is
well-known, and the following result dates from the 70s: see Mohan, \cite{M}%
, who improves earlier results in \cite{T}.%
\begin{equation}
U(x)-m^{-1}x-m^{-1}\int_{0}^{x}\overline{\Phi }(y)dy=o(\int_{0}^{x}\overline{%
\Phi }(y)dy)\text{ as }x\rightarrow \infty .  \label{3}
\end{equation}%
Later Sgibnev showed, in \cite{S2}, that (\ref{3}) actually holds whenever $%
m $ is finite and $EX_{1}^{2}=\infty ,$ so that the assumption of a
regularly varying tail is redundant. This in turn suggests that if we do
make this assumption we should be able to improve on (\ref{3}). Under our
assumptions $\overline{\Phi }\in RV(-\beta ),$ where $\beta =\alpha -1,$ so
any statement that the LHS of (\ref{3}) is $O(x^{\gamma })$ with $\gamma
<1-\beta $ would be an improvement. In fact we can be much more precise than
this.

We write $\phi _{2}$ for the convolution $\phi \ast \phi $ and define
real-valued functions $g$ \ and $\overline{G}$ on $[0,\infty )$ by 
\begin{eqnarray}
g(y) &=&2\phi (y)-\phi _{2}(y),  \label{2} \\
\overline{G}(x) &=&\int_{x}^{\infty }g(z)dz,\text{ so that }\overline{G}%
(0)=\int_{0}^{\infty }g(z)dz=1.  \label{14}
\end{eqnarray}

To state our result, we set%
\begin{equation}
U(x)-m^{-1}x-m^{-1}\int_{0}^{x}\overline{\Phi }(y)dy=m^{-1}V(x),  \label{15}
\end{equation}%
so that the known result (\ref{3}) says that $V(x)=o(\overline{\overline{%
\Phi }}(x)),$ where $\overline{\overline{\Phi }}(x):=\int_{0}^{x}\overline{%
\Phi }(y)dy\in RV(1-\beta ).$

\begin{theorem}
\label{M} Take $\alpha \in (1,2)$ and $\beta =\alpha -1.$

(i) Define a constant by%
\begin{equation*}
c_{\alpha }=(1-2\beta )\int_{0}^{1}\frac{dw}{w^{\beta }(1-w)^{\beta }}=\frac{%
\Gamma (1-\beta )^{2}}{\Gamma (1-2\beta )}.
\end{equation*}%
Then%
\begin{equation*}
\lim_{x\rightarrow \infty }\frac{\overline{G}(x)}{\overline{\Phi }(x)^{2}}%
=c_{\alpha }.
\end{equation*}%
(ii) The asymptotic behaviour of $V$ is given by%
\begin{eqnarray}
V(x) &\backsim &\frac{|c_{\alpha }|x\overline{\Phi }(x)^{2}}{|2\beta -1|}%
\text{ if }\beta \neq 1/2,  \label{m2} \\
\text{ }V(x) &\rightarrow &\int_{0}^{\infty }\overline{G}(y)dy\text{ if }%
\beta =1/2\text{ and}\int_{0}^{\infty }\overline{\Phi }(y)^{2}dy<\infty ,
\label{m4} \\
V(x) &=&o\left( \int_{0}^{x}\overline{\Phi }(y)^{2}dy\right) \text{if }\beta
=1/2\text{ and}\int_{0}^{\infty }\overline{\Phi }(y)^{2}dy=\infty .
\label{m3}
\end{eqnarray}
\end{theorem}

\begin{remark}
Since $\overline{\Phi }(x)^{2}\in RV(-1)$ when $\beta =1/2$ we see that in (%
\ref{m3}) $\int_{0}^{x}\overline{\Phi }(y)^{2}dy$ is slowly varying. Also in
(\ref{m4}) $\int_{0}^{\infty }\overline{G}(y)dy=0$ iff%
\begin{equation}
\frac{1-\hat{\phi}(\lambda )}{\sqrt{\lambda }}=\frac{\int_{0}^{\infty
}(1-e^{-\lambda x})\overline{F}(x)}{m\sqrt{\lambda }}\rightarrow 0\text{ as }%
\lambda \downarrow 0.  \label{m5}
\end{equation}
\end{remark}

\begin{remark}
We cannot give the exact behaviour of $V$ when $\alpha =2,$ but it is not
difficult to show that in this case $V(x)=o(x^{\varepsilon -1})$ for any
fixed $\varepsilon >0.$
\end{remark}

\section{Proofs}

(i) Recall that $\phi (x)=m^{-1}\overline{F}(x)$ is decreasing, bounded and
is in $RV(-\alpha ).$ Then write 
\begin{eqnarray*}
\overline{G}(x) &=&\int_{x}^{\infty }(2\phi (y)-\int_{0}^{y}\phi (y-w)\phi
(w)dw)dy \\
&=&\int_{x}^{\infty }(2\phi (y)\int_{0}^{y/2}\phi (w)dw-2\int_{0}^{y/2}\phi
(y-w)\phi (w)dw)dy+2\int_{x}^{\infty }\phi (y)\overline{\Phi }(y/2)dy. \\
&:&=I_{1}+I_{2}.
\end{eqnarray*}%
Since $\overline{\Phi }(y/2)\backsim 2^{\beta }\overline{\Phi }(y),$ we see
that $I_{2}$ $\backsim 2^{\beta }\overline{\Phi }(x)^{2}.$ Also 
\begin{eqnarray*}
-I_{1} &=&2\int_{x}^{\infty }dy\int_{0}^{y/2}(\phi (y-w)-\phi (y))\phi (w)dw
\\
&=&2\int_{0}^{\infty }\phi (w)dw\int_{2w\vee x}^{\infty }(\phi (y-w)-\phi
(y))dy \\
&=&2\int_{0}^{x/2}\phi (w)(\overline{\Phi }(x-w)-\overline{\Phi }%
(x))dw+2\int_{x/2}^{\infty }\phi (w)(\overline{\Phi }(w)-\overline{\Phi }%
(2w))dw.
\end{eqnarray*}%
As $\overline{\Phi }(w)-\overline{\Phi }(2w)\backsim (1-2^{-\beta })%
\overline{\Phi }(w)$ we see that the second term is asymptotic to $%
(1-2^{-\beta })\overline{\Phi }(x/2)^{2},$ or equivalently $2^{\beta
}(2^{\beta }-1)\overline{\Phi }(x)^{2}.$ Also we can write the first term as 
\begin{equation*}
2\int_{0}^{x/2}\phi (w)dw\int_{x-w}^{x}\phi (y)dy=2(x\phi
(x))^{2}\int_{0}^{1/2}\frac{\phi (xw)}{\phi (x)}dw\int_{1-w}^{1}\frac{\phi
(xy)}{\phi (x)}dy.
\end{equation*}%
For $y\in (1-w,1]$ and $w\in (0,1/2],$ we have%
\begin{equation*}
\frac{\phi (xw)}{\phi (x)}\frac{\phi (xy)}{\phi (x)}\overset{\text{uniformly}%
}{\rightarrow }w^{-\alpha }y^{-a},
\end{equation*}%
and $x\phi (x)\backsim \beta \overline{\Phi }(x),$ so this is asymptotic to $%
2(\beta \overline{\Phi }(x))^{2}J_{\alpha },$ where%
\begin{eqnarray*}
J_{\alpha } &=&\int_{0}^{1/2}w^{-\alpha }dw\int_{1-w}^{1}y^{-a}dy=\beta
^{-1}\int_{0}^{1/2}w^{-a}\{(1-w)^{-\beta }-1\}dw \\
&=&-\beta ^{-2}2^{\beta }(2^{\beta }-1)+\beta ^{-1}\int_{0}^{1/2}w^{-\beta
}(1-w)^{-\alpha }dw.
\end{eqnarray*}

This establishes the result, and gives%
\begin{equation}
c_{\alpha }=2^{2\beta }-2\beta I_{\alpha },\text{ where }I_{\alpha
}=\int_{0}^{1/2}\frac{dw}{(1-w)\{w(1-w)\}^{\beta }}.
\end{equation}%
But 
\begin{eqnarray*}
I_{\alpha } &=&\int_{0}^{1/2}\frac{w+(1-w)dw}{(1-w)\{w(1-w)\}^{\beta }}%
=\int_{0}^{1/2}\frac{w^{1-\beta }dw}{(1-w)^{1+\beta }}+\int_{0}^{1/2}\frac{dw%
}{w^{\beta }(1-w)^{\beta }} \\
&=&\beta ^{-1}2^{2\beta -1}+\beta ^{-1}(\beta -1)\int_{0}^{1/2}\frac{%
w^{1-\beta }dw}{(1-w)^{\beta }}+\int_{0}^{1/2}\frac{dw}{w^{\beta
}(1-w)^{\beta }} \\
&=&\beta ^{-1}2^{2\beta -1}+(1-(2\beta )^{-1})\int_{0}^{1}\frac{dw}{w^{\beta
}(1-w)^{\beta }},\text{ }
\end{eqnarray*}%
so $c_{\alpha }=(1-2\beta )B(1-\beta ,1-\beta )$as required.

(ii) We start by noting that the stationarity of $\phi $ gives $%
\int_{0}^{x}\phi (x-y)U(y)dy=m^{-1}x,$ and then 
\begin{eqnarray*}
\int_{0}^{x}\phi _{2}(x-y)U(y)dy &=&\int_{0}^{x}\int_{0}^{x-y}\phi
(x-y-z)\phi (z)dzU(y)dy \\
&=&\int_{0}^{x}\phi (z)dz\int_{0}^{x-z}\phi (x-y-z)U(y)dy \\
&=&m^{-1}\int_{0}^{x}(x-z)\phi (z)dz=m^{-1}(x-\int_{0}^{x}\overline{\Phi }%
(y)dy).
\end{eqnarray*}%
Thus 
\begin{equation*}
\int_{0}^{x}g(x-y)U(y)dy=m^{-1}(x+\int_{0}^{x}\overline{\Phi }(y)dy),
\end{equation*}%
and 
\begin{equation}
m^{-1}V(x)=U(x)-m^{-1}(x+\int_{0}^{x}\overline{\Phi }(y)dy)=U(x)-%
\int_{0}^{x}g(x-y)U(y)dy,  \label{5}
\end{equation}%
and integration by parts gives%
\begin{equation}
V(x)=m\int_{[0,x)}\overline{G}(x-y)U(dy).  \label{b}
\end{equation}

Although statement (\ref{m2}) unifies the cases $\beta \in (0,1/2)$ and $%
\beta \in (1/2,1)$ their proofs differ. In the first case $\int_{0}^{x}%
\overline{\Phi }^{2}(y)dy\rightarrow \infty ,$ and we can use (\ref{b}) in
conjunction with the following, which is Theorem 4 in \cite{S2}, and \ shows
that Theorem 2.1 in \cite{M} holds without assuming asymptotic stability.

\begin{lemma}
(Sgibnev)\label{C} Let $Q$ be a non-negative, non-increasing bounded
function and put $A(x)=\int_{0}^{x}Q(y)dy.$ Then if $A(\infty )=\infty .$ 
\begin{equation}
\int_{0}^{x}Q(x-y)dU(y)\backsim m^{-1}A(x)\text{ as }x\rightarrow \infty .
\label{13}
\end{equation}
\end{lemma}

If $\beta \in (0,1/2)$ we have $c_{\alpha }>0,$ so given $\varepsilon >0$ $%
\exists x_{0}$ such that for all $x>x_{0}$ 
\begin{equation}
(c_{\alpha }-\varepsilon )Q(x)\leq \overline{G}(x)\leq (c_{\alpha
}+\varepsilon )Q(x),  \label{n}
\end{equation}%
where $Q(x)$ $=\overline{\Phi }^{2}(x)$ satisfies the conditions of Lemma %
\ref{C}. Since the contribution to the integral in (\ref{b}) from $[0,x_{0}]$
is $O(\overline{G}(x)),$ which is neglible, it follows that 
\begin{equation*}
m^{-1}V(x)\thicksim m^{-1}\int_{0}^{x}Q(y)dy,\text{ so }V(x)\thicksim \frac{%
c_{\alpha }x\overline{\Phi }^{2}(x)}{1-2\beta },
\end{equation*}%
and (\ref{m2}) holds. If $\beta =1/2$ we have $c_{a}=0$ but (\ref{n}) still
holds and provided $\int_{0}^{\infty }Q(y)dy=\infty $ the conditions of
Lemma \ref{C} are satisfied and the proof of (\ref{m3}) follows. In the
remaining cases it is clear that $\overline{G}$ is Directly Riemann
Integrable, so the Key Renewal Theorem applies to (\ref{b}) to give $%
V(x)\rightarrow \int_{0}^{\infty }\overline{G}(y)dy,$ and we need only show
when this is $0.$ From (\ref{2}) we see that the ordinary Laplace transforms
of $\phi $ and $g$ are related by%
\begin{equation*}
1-\hat{g}(\lambda )=(1-\hat{\phi}(\lambda ))^{2}\thicksim \lambda ^{2\beta
}L(\lambda )\text{ as }\lambda \rightarrow 0,
\end{equation*}%
where $L$ is slowly varying at zero, so we have $(1-\hat{g}(\lambda
))/\lambda \rightarrow 0$ as $\lambda \rightarrow 0$ iff $\beta >1/2$ or $%
\beta =1/2$ and (\ref{m5}) holds. But since $g$ is bounded in absolute value
by the integrable function $2\phi +\phi _{2},$ we can interchange orders of
integration to see that 
\begin{equation*}
(1-\hat{g}(\lambda ))/\lambda =\int_{0}^{\infty }e^{-\lambda x}\overline{G}%
(x)dx,\text{ }
\end{equation*}%
and the conclusion follows by letting $\lambda $ go to $0.$

For the case $\beta \in (1/2,1)$ we write $g^{\ast },\overline{G^{\ast }}$
for $-g,-\overline{G},$ and we claim first that $\overline{G^{\ast }}$ is
eventually positive and monotone, which follows from the fact%
\begin{equation}
\lim \inf_{x\in \infty }\frac{g^{\ast }(x)}{2x\phi (x)^{2}}\geq \frac{%
-c_{\alpha }}{\beta }>0.  \label{6}
\end{equation}%
To see that (\ref{6}) holds, write%
\begin{eqnarray*}
g^{\ast }(x) &=&2\left( \int_{0}^{x/2}\phi (w)\{\phi (x-w)-\phi (x)\}dw-\phi
(x)\overline{\Phi }(x/2)\right) \\
&=&2x\phi (x)^{2}\left( \int_{0}^{1/2}\frac{\phi (w)}{\phi (x)}\{\frac{\phi
(x-xw)}{\phi (x)}-1\}dw-\frac{\overline{\Phi }(x/2)}{x\phi (x)}\right) .
\end{eqnarray*}%
Since the integrand converges pointwise to $w^{-\alpha }\{(1-w)^{-\alpha
}-1\}$ it follows from Fatou's Lemma that 
\begin{eqnarray*}
\lim \inf_{x\in \infty }\frac{g^{\ast }(x)}{2x\phi (x)^{2}} &\geq
&\int_{0}^{1/2}w^{-\alpha }\{(1-w)^{-\alpha }-1\}dw-\beta ^{-1}2^{\beta } \\
&=&I_{\alpha }+\beta J_{\alpha }-\beta ^{-1}2^{\beta }=\frac{-c_{\alpha }}{%
\beta },
\end{eqnarray*}%
as claimed. So we can fix $x_{0}$ so that $g^{\ast }(x)>0$ for $x>x_{0},$
and then, as in the above referenced proof in \cite{S2}, given any $%
\varepsilon >0$ we can find $x_{1}>x_{0}$ with%
\begin{eqnarray*}
\int_{x_{1}}^{x}\overline{G^{\ast }}(x-y)dU(y) &\leq &\frac{1+\varepsilon }{m%
}\int_{x_{1}}^{x}\overline{G^{\ast }}(x-y)dy \\
&=&\frac{1+\varepsilon }{m}\int_{x-x_{1}}^{\infty }\overline{G}(z)dz\backsim 
\frac{1+\varepsilon }{m}\frac{c_{\alpha }x\overline{\Phi }(x)^{2}}{2\beta -1}%
,
\end{eqnarray*}%
where we have used $\int_{0}^{\infty }\overline{G}(z)dz=0,$ and $%
\int_{x-x_{1}}^{x}\overline{G}(z)dz=O(\overline{\Phi }(x)^{2}).$ Using a
corresponding lower bound and the fact that $\int_{[0,x_{1})}\overline{%
G^{\ast }}(x-y)dU(y)=O(\overline{\Phi }(x)^{2}),$ (\ref{m2}) follows.

\section{The Random walk case}

If the variables $X_{1},X_{2},\cdots $ can take positive and negative
values, we will still define the renewal measure by (\ref{4}), and study $%
U(x)=U([0,x])$ as $x\rightarrow \infty .$ (For a different interpretation of
the renewal function see \cite{S1}.) In this case it is also shown in \cite%
{S2} that (\ref{3}) holds only assuming $m=EX_{1}\in (0,\infty )$ and $%
E(X_{1}^{+})^{2}=\infty .$ The idea of that proof is to express $U$ in terms
of $U^{\uparrow },$ and $U^{\downarrow },$ the renewal measures for the
process of increasing and decreasing ladder heights, and then use (\ref{3})
for $U^{\uparrow }.$ We will use a similar argument to give an extension of
(ii) of our Theorem \ref{M} to the random walk case.

To clarify, if $\tau _{n}$is the $n^{\text{th}}$ strict increasing ladder
epoch and $\sigma _{n}$ is the $n^{\text{th}}$ weak decreasing ladder epoch
(with $\tau _{0}=\sigma _{0}=0)$, we put%
\begin{eqnarray*}
U^{\uparrow }(dx) &=&\sum_{0}^{\infty }P(H_{n}^{\uparrow }\in dx),\text{
where }H_{n}^{\uparrow }=S_{\tau _{n}}\text{ and } \\
U^{\downarrow }(dx) &=&\sum_{0}^{\infty }P(H_{n}^{\downarrow }\in dx),\text{
where }H_{n}^{\downarrow }=|S_{\sigma _{n}}|.\text{ }
\end{eqnarray*}
Since $m>0$ we know that $H_{1}^{\downarrow }$ is improper and $%
U^{\downarrow \text{ }}$ is a finite measure. Everything depends on the
following simple observation:

\begin{lemma}
We have%
\begin{equation}
U(dx)=\int_{0}^{\infty }U^{\downarrow \text{ }}(dy)U^{\uparrow }(y+dx),\text{
}x>0,  \label{19}
\end{equation}
\end{lemma}

\begin{proof}
Since the Fourier transforms of the measures $U,U^{\uparrow }$ and $%
U^{\downarrow \text{ }}$ are $(1-E(e^{i\theta S_{1}}))^{-1},(1-E(e^{i\theta
H_{1}^{\uparrow }}))^{-1},$ and $(1-E(e^{i\theta H_{1}^{\downarrow
}}))^{-1}, $ This is immediate from the Wiener-Hopf factorisation.
\end{proof}

\begin{remark}
This paraphrases the Lemma on p 790 of \cite{S2}.
\end{remark}

A further consequence of the Wiener-Hopf factorisation is that%
\begin{equation*}
C:=\int_{0}^{\infty }U^{\downarrow \text{ }}(dy)=\frac{m^{\uparrow }}{m},%
\text{ where }m^{\uparrow }=EH_{1}^{\uparrow }.
\end{equation*}%
Moreover the duality lemma gives, as $z\rightarrow \infty $ 
\begin{eqnarray}
\overline{\Phi ^{\uparrow }}(z) &:&=\frac{1}{m^{\uparrow }}\int_{z}^{\infty
}P(H_{1}^{\uparrow }>w)dw=\frac{1}{m^{\uparrow }}\int_{0}^{\infty
}U^{\downarrow \text{ }}(dy)\int_{z}^{\infty }P(S_{1}>w)dw  \notag \\
&\backsim &\frac{C\int_{z}^{\infty }P(S_{1}>w)dw}{m^{\uparrow }}=\overline{%
\Phi }(z).  \label{17}
\end{eqnarray}

\begin{remark}
Actually what is shown in \cite{S2} is that 
\begin{equation}
U(x)-m^{-1}x\backsim m^{-1}\int_{0}^{x}\overline{\Phi ^{\uparrow }}(y)dy,%
\text{ }  \label{18}
\end{equation}%
and then a version of (\ref{17}) is used to obtain (\ref{3}). But in
examining the remainder it is important that we use (\ref{18}).
\end{remark}

Our extension of Theorem \ref{M} is

\begin{theorem}
\label{N}Assume that $ES_{1}=m\in (0,\infty )$ and $\overline{F}\in
RV(-\alpha )$ with $\alpha \in (1,2).$ Write$\overline{\text{ }\Phi
^{\uparrow }}$ and $\overline{G^{\uparrow }}$ for the functions$\overline{%
\text{ }\Phi }$ and $\overline{G}$ evaluated for the renewal process $%
(H_{n}^{\uparrow },n\geq 0),$ and set%
\begin{eqnarray*}
\Psi (x) &=&\frac{1}{m^{\uparrow }}\int_{0}^{\infty }U^{\downarrow
}(dy)\int_{y}^{x+y}\overline{\text{ }\Phi ^{\uparrow }}(z)dz-K,\text{ where}
\\
K &=&0\text{ if }\int_{0}^{\infty }\overline{\Phi }(y)^{2}dy=\infty ,\text{ }%
K=\int_{0}^{\infty }U^{\downarrow }(dy)V^{\uparrow }(y)\text{ if }%
\int_{0}^{\infty }\overline{\Phi }(y)^{2}dy<\infty .
\end{eqnarray*}%
and%
\begin{equation*}
m^{-1}\tilde{V}(x)=U(x)-\frac{x}{m}-\Psi (x),
\end{equation*}%
Then we have that the statements (\ref{m2}), (\ref{m4}) and (\ref{m3}) of
Theorem \ref{M} hold with $V$ replaced by $\tilde{V}$ .
\end{theorem}

\begin{proof}
From (\ref{18}) we have%
\begin{equation*}
U(x)=\int_{0}^{\infty }U^{\downarrow \text{ }}(dy)\left( U^{\uparrow
}(y+x)-U^{\uparrow }(y)\right) ,
\end{equation*}%
so that if we substitute (\ref{15}) for $U^{\uparrow }$ we get%
\begin{eqnarray*}
U(x) &=&\frac{1}{m^{\uparrow }}\int_{0}^{\infty }U^{\downarrow \text{ }%
}(dy)\left( x+\int_{y}^{x+y}\overline{\Phi ^{\uparrow }}(z)dz+V^{\uparrow
}(x+y)-V^{\uparrow }(y)\right) \\
&=&\frac{Cx}{m^{\uparrow }}+\Psi (x)+\frac{1}{m^{\uparrow }}\int_{0}^{\infty
}U^{\downarrow \text{ }}(dy)(V^{\uparrow }(x+y)-V^{\uparrow }(y)) \\
&:&=\frac{x}{m}+\Psi (x)+\frac{I(x)}{m^{\uparrow }},
\end{eqnarray*}%
and we need to examine the behaviour of $I(x).$ Note that for $\beta >1/2$
we have $\int_{0}^{\infty }U^{\downarrow \text{ }}(dy)V^{\uparrow }(y)$
finite, and $\int_{0}^{\infty }U^{\downarrow \text{ }}(dy)V^{\uparrow
}(x+y)\backsim CV^{\uparrow }(x).$ For $\beta <1/2$ we have $V^{\uparrow
}(x)\rightarrow \infty $ and 
\begin{equation*}
\frac{V^{\uparrow }(x+y)-V^{\uparrow }(y)}{V^{\uparrow }(x)}\rightarrow 1,%
\text{ }
\end{equation*}%
and we can modify the argument in \cite{S2} to \ show that dominated
convergence applies to give the result. Similar arguments deal with the case 
$\beta =1/2.$\ 
\end{proof}

\section{Concluding remarks}

It is easy to see that in the renewal case we can expand $\int_{0}^{\infty
}e^{-\lambda x}U(x)dx$ in powers of $1-\hat{\phi}(\lambda )$ as follows:%
\begin{equation*}
\hat{U}(\lambda )=\frac{1}{\lambda }+\frac{1}{m\lambda ^{2}}\left(
1+\sum_{1}^{\infty }(1-\hat{\phi}(\lambda ))^{r}\right)
\end{equation*}%
Now 
\begin{eqnarray*}
\frac{(1-\hat{\phi}(\lambda ))}{m\lambda ^{2}} &=&m^{-1}\int_{0}^{\infty
}e^{-\lambda x}\int_{0}^{x}\overline{\Phi }(y)dy, \\
\frac{(1-\hat{\phi}(\lambda ))^{2}}{m\lambda ^{2}} &=&m^{-1}\int_{0}^{\infty
}e^{-\lambda x}\int_{0}^{x}\overline{G}(y)dy,
\end{eqnarray*}%
and in fact for any $r\geq 2$%
\begin{equation*}
\frac{(1-\hat{\phi}(\lambda ))^{r}}{m\lambda ^{2}}=m^{-1}\int_{0}^{\infty
}e^{-\lambda x}\int_{0}^{x}\overline{G_{r}}(y)dy,
\end{equation*}%
where $\overline{G_{r}}(y)=\int_{y}^{\infty }g_{r}(z)dz$ and the sequence of
functions $g_{r}$ are defined by 
\begin{equation*}
g_{2}=g=2\phi -\phi \ast \phi \text{ and }g_{r+1}=\phi +g_{r}-\phi \ast
g_{r},\text{ }r\geq 2.
\end{equation*}%
If one could justify inverting the transform, writing $\overline{G_{1}}$ for 
$\overline{\Phi },$ this would yield a complete asymptotic expansion 
\begin{equation*}
U(x)=1+\frac{x}{m}+\frac{1}{m}\sum_{1}^{\infty }\int_{0}^{x}\overline{G_{r}}%
(y)dy,
\end{equation*}%
and our results involve only the first two terms in the sum. The crux of our
result is the justification of the relation $\overline{G_{2}}(y)\backsim
c_{\alpha }\overline{\Phi }(x)^{2},$ so a natural question is whether one
can show that $\overline{G_{r}}(y)\backsim c\overline{\Phi }(x)^{r}.$ This
seems to be impossible without making extra assumptions, but it seems that
the not unnatural assumption that $F$ has a monotone density would permit
verification of this when $r=3.$ This would then give an extra term in our
result when $\beta <1/2.$

\end{document}